\documentclass[11pt,leqno]{article}
\usepackage{amsmath, amsfonts, theorem,color, 
}
\usepackage{latexsym}
\usepackage{graphicx}
\makeatletter
\def\@maketitle{\newpage
    \null
    \vskip .8truein
    \begin{center}%
     {\bf \@title \par}%
     \vskip 1.5em
     {\small
      \lineskip .5em
      \begin{tabular}[t]{c}\@author
      \end{tabular}\par}%
    \end{center}%
    \par
    \vskip .4truein}
\@addtoreset{equation}{section} \@addtoreset{theorem}{section}
\@addtoreset{lemma}{section} \@addtoreset{proposition}{section}
\@addtoreset{definition}{section} \@addtoreset{corollary}{section}
\@addtoreset{remark}{section}

\newcommand{\re}{{\mathbb R}}

\newtheorem{theorem}{Theorem}[section]
\newtheorem{lemma}{Lemma}[section]
\newtheorem{proposition}{Proposition}[section]

\newtheorem{remark}{Remark}[section]

  {\hfill$\Box$\bigskip\par}

\def\proof{\list{}{\setlength{\leftmargin}{0pt}
                      \parskip=0pt\parsep=0pt\listparindent=2em
                      \itemindent=0pt}\item[]\futurelet\testchar\@maybe}

\def\@maybe{\ifx[\testchar \let\next\@Opt
          \else \let\next\@NoOpt \fi \next}
\def\@Opt[#1]{{\it Proof of #1.\ }}\def\@NoOpt{{\it Proof.\ }}
\begin{document}
\title{\Large \bf Asymptotic behaviour for operators of Grushin type: invariant measure and singular perturbations.}

\author{{\large \sc Paola Mannucci, Claudio Marchi, Nicoletta Tchou}\\
 \rm Universit\`a degli Studi di Padova, Universit\'e de Rennes 1}
\maketitle

\begin{abstract}
\noindent This paper concerns singular perturbation problems where the dynamics of the fast variable evolve in the whole space according to an operator whose infinitesimal generator is formed by a Grushin type second order part and a Ornstein-Uhlenbeck first order part. \\
We prove that the dynamics of the fast variables admits an invariant measure and that the associated ergodic problem has a viscosity solution which is also regular and with logarithmic growth at infinity. These properties play a crucial role in the main theorem which establishes that the value functions of the starting perturbation problems converge to the solution of an effective problem whose operator and initial datum are given in terms of the associated invariant measure.
\end{abstract}
\footnote{\date{\today}}



\section{Introduction}

This paper is devoted to study with PDE's methods some asymptotic features of processes described by the dynamics 
\begin{equation}\label{process}
dZ_t=b(Z_t)dt +\sqrt2\sigma(Z_t)dW_t\quad \textrm{for }t\in(0,+\infty),\qquad Z_0= x^0\in \re^2,
\end{equation}
where $W_t$ is a $2$-dimensional Brownian motion while the matrix $\sigma$ is degenerate and of Grushin type and the drift~$b$ is of Ornstein-Uhlenbeck type, namely 
\begin{equation}\label{sigma}
\sigma(x)=\left(\begin{array}{cc} 1&0\\0&x_1\end{array}\right), \ 
 b(x)=-\alpha x,\  \alpha>0.
\end{equation}
The columns of $\sigma$ in \eqref{sigma} satisfy H\"ormander condition: $X_1=(1,0)$ and $[X_1,X_2]=(0,1)$ span all $\re^2$. Hence, we have that $[X_1, X_2]=\partial_{x_2}$.\\
In particular, we shall investigate:\\
1) existence and uniqueness of the invariant measure of this process;\\
2) existence, uniqueness and regularity of the solution for the ergodic problem of the infinitesimal generator~${\cal L}$  (see \eqref{op_l});\\
3) the asymptotic behaviour as $\epsilon \to 0$ of the value function of optimal control problems driven by
\begin{equation}\label{dynXY}
\left\{\begin{array}{l}
dX_t=\tilde\phi (X_t,Y_t,u_t)dt+\sqrt2 \tilde \sigma(X_t,Y_t,u_t)dW_t,\qquad X_0=x\in \re^n\\
dY_t= \frac1\epsilon b(Y_t)dt+\frac{\sqrt2}{\sqrt\epsilon}\sigma(Y_t) dW_t,\qquad Y_0=y\in\re^2.
\end{array}\right.
\end{equation}

The paper is organized as follows: in Section 2 we prove existence and uniqueness of the invariant measure. In Section 3 we establish our main result on perturbation problem, to this end, we introduce the approximated ergodic problems and investigate the regularity of their solutions.

\section{Existence of the invariant measure}\label{sect:measure}
We consider the stochastic dynamics \eqref{process} with coefficients as in~\eqref{sigma}. The main aim of this section  is to prove existence and uniqueness of the invariant measure~$m$ associated to the process~\eqref{process}. To this goal, we use a Liouville property for the infinitesimal generator of~\eqref{process} (see~\cite{CDC} for other Liouville properties for Grushin operator in a semilinear framework with a superlinear growth for the zeroth order term).\\

Let us recall from \cite{B1} that a probability measure~$m$ on $\re^2$ is an {\it invariant measure} for process~\eqref{process} if, for each $u_0\in \mathbb L^\infty(\re^2)$, it satisfies
\begin{equation*}
\int_{\re^2} u(x,t)m(x)\, dx = \int_{\re^2} u_0(x)m(x)\, dx
\end{equation*}
where $u(x,t)=\mathbb E_x(u_0(X_t))$ is the solution to the parabolic Cauchy problem
\begin{equation*}
\partial_t u+{\cal L}u=0\quad \textrm{in }(0,+\infty)\times \re^2,\qquad u(x,0)=u_0(x)\quad \textrm{on }\re^2
\end{equation*}
where
\begin{equation}\label{op_l}
{\cal L}(y,q,Y):=- tr(\sigma\sigma^T Y)-b\cdot q\equiv -Y_{11}-y_1^2V_{22}-b(y)q
\end{equation}
is the {\it infinitesimal generator} of process~\eqref{process}.
For the sake of completeness, let us recall the result in \cite[Example 5.1]{MMT1}.

\begin{theorem}\label{mgen}
The diffusion process \eqref{process} admits exactly one invariant probability measure $m$.
\end{theorem}
\begin{proof}
Under our assumptions it is easy to check that the matrix 
$A_{\rho}=\sigma_{\rho}\sigma^T_{\rho}$ where
$$
\sigma_{\rho}(x)=\left(\begin{array}{ccc} 1&0&0\\0&x_1&\rho\end{array}\right)
$$
is such that  
$A_{\rho}\rightarrow A \textrm{ in } L^{\infty}$ and  $A_{\rho}$ is locally definite positive.
Moreover taking
$$W(x)=\frac{1}{12}x_1^4+\frac{1}{2}x_2^2,$$
we have
$$-tr(A_{\rho}(x)D^2 W(x)))=
- 2x_1^2-\rho^2.$$
$$-b(x)DW=
\frac{1}{3}\alpha x_1^4+\alpha x_2^2.$$
Then $\mathcal L_{\rho}W=-tr(\sigma_{\rho}\sigma^T_{\rho}D^2u)+b\cdot Du\geq 1$ is equivalent to the following condition
\begin{equation*}
\frac{1}{3}\alpha x_1^4+\alpha x_2^2 \geq  2 x_1^2+1+\rho^2,
\end{equation*}
hence
$W$
satisfies 
\begin{equation*}
\mathcal L_{\rho}W\geq 1 \quad \textrm{in }\overline{B(0,R_0)}^C,\quad
W\geq 0\quad \textrm{in } \overline{B(0,R_0)}^C,\quad \lim_{\vert x\vert\rightarrow \infty} W=\infty
\end{equation*}
for $\rho$ sufficiently small.

Then following the procedure used in \cite[Proposition 2.1]{MMT1}, using the function $W$, there exists an unique invariant measure $m_{\rho}$ for the process with diffusion $\sigma_{\rho}$, and arguing as in \cite[Theorem 2.1 (proof)]{MMT1} and 
using again the function $W$ we obtain the existence of the invariant measure associated to the process \eqref{process}.
For the proof of uniqueness, we refer the reader to \cite[Theorem 2.1]{MMT1}.
\end{proof}
\begin{remark}\rm {
Lions-Musiela in \cite{LM} have considered a similar degenerate case
but in their paper the elements of the matrix are bounded in $\re^2$ in this way
\begin{equation*}
\sigma(x)=\left(\begin{array}{cc} 1&0\\0&\frac{x_1}{\sqrt{1+x_1^2}}\end{array}\right).
\end{equation*}  }
\end{remark}
\begin{remark}{\rm
Following \cite[Example 5.1]{MMT1} can can obtain similar results of Theorem \ref{mgen} with a more general drift term $b$:
\begin{equation*}
  b_i(x)=  b_i(x_i),\qquad
\left\{\begin{array}{ll}
b_i(x_i)\leq-\frac{C_i}{|x_i|^{1-\beta}} &\textrm{for }x_i\geq R\\
b_i(x_i)\geq\frac{C_i}{|x_i|^{1-\beta}} &\textrm{for }x_i\leq -R
\end{array}\right.
\end{equation*}
for $\beta\geq 0$, $R>0$ and suitably chosen $C_i>0$ ($i=1,2$).}
\end{remark}
\begin{remark}
\label{ergodics}
{\rm
As applications of the existence of an invariant measure we obtain, arguing as in \cite{MMT1}
the following results:
$$
\lim_{\delta\rightarrow 0^+}\delta u_{\delta}(x)=  
\lim_{t\to +\infty}u(t, x)=
\lim_{t\to +\infty}\frac{v(t,x)}{t}=
\int_{\re^2}f dm,$$
where $m$ is the invariant measure of  and $u_{\delta}$, $u$ and $v$ are the solutions respectively of
$$\delta u_{\delta}+{\mathcal L}u_{\delta}=f(\cdot),  \qquad\textrm{in }\re^2,$$
$$
u_t+{\mathcal L}u=0\qquad\textrm{in }(0,+\infty)\times\re^2,\qquad u(0,\cdot)=f(\cdot)  \qquad \textrm{on }\re^2,$$
$$
v_t+{\mathcal L}v=f(\cdot)\qquad \textrm{in }(0,+\infty)\times\re^2,\qquad v(0,\cdot)=0  \qquad \textrm{on }\re^2,$$
and ${\mathcal L}$ is defined in (\ref{op_l}).}
\end{remark}

\section{Asymptotic behaviour for a singular perturbation problem}\label{sect:SP}
In this section,
we investigate the limit of the value function
\begin{equation*}
V^\epsilon(t,x,y):=\sup_{u \in{\cal U}} \mathbb E[\int_t^Tf(X_s,Y_s,u_s)ds+ e^{a(t-T)}g(X_T)]
\end{equation*}
where $\mathbb E$ denotes the expectation, $\cal U$ is the set of progressively measurable processes with values in a compact metric set ~$U$ and $a$ is a fixed positive parameter and $(X_t,Y_t)$ are driven by \eqref{dynXY} (note that $V^\epsilon$ depends on $\epsilon$ through the coefficients of the dynamics).\\
Throughout this section, we shall assume
\begin{itemize}
\item[$i$)] the function~$f=f(x,y,u)$ is Lipschitz continuous in~$(x,y)$ uniformly in $u$ and, for some $C_f>0$, it satisfies
\begin{equation*}
|f(x,y,u)|\leq C_f(1+|x|)\qquad \forall (x,y,u)\in \re^n\times \re^2\times U;
\end{equation*}
\item[$ii$)] 
the function~$g$ is continuous in~$(x,y)$ and there exits $C_g$ such that
\begin{equation*}
|g(x,y)|\leq C_g(1+|x|) \qquad \forall (x,y)\in \re^n\times\re^2;
\end{equation*}
\item[$iii$)] $\tilde \phi(x,y,u)$ and $\tilde\sigma(x,y,u)$ are Lipschitz continuous and bounded in $(x,y)$ uniformly on $u$:
 $\vert\tilde \phi(x,y,u)\vert\leq C_{\tilde \phi}$,  $\vert\tilde \sigma(x,y,u)\vert\leq C_{\tilde\sigma}$.
\end{itemize}

Problems of this type arise from models where the variables $Y$ evolve much faster than the variables $X$. We refer to \cite{BCM} and \cite{MMT2} for the financial models which inspired this research.

By standard theory (see~\cite{FS}), the value function $V^\epsilon$ is the unique (viscosity) solution to the following Cauchy problem
\begin{equation}\label{HJBSP}
\left\{\begin{array}{ll}
-\partial_tV^\epsilon+H\left(x,y,D_xV^\epsilon,D_{xx}^2V^\epsilon,\frac{D_{xy}^2V^\epsilon}{\sqrt\epsilon}\right)&\\
\qquad+\frac1\epsilon{\cal L}(y, D_yV^\epsilon,D_{yy}V^\epsilon)+a V^\epsilon=0 & \textrm{in }(0,T)\times \re^n\times\re^2\\
V^\epsilon(T,x,y)=g(x,y)& \textrm{on }\re^n\times\re^2
\end{array}\right.
\end{equation}
where ${\cal L}$ is the operator defined in~\eqref{op_l} and
\begin{eqnarray*}
H(x,y,p,X,Z)&:=&\min_{u\in U}\left\{-tr(\tilde \sigma\tilde \sigma^T X)-\tilde \phi\cdot p -2 tr(\tilde \sigma\sigma^T Z)-f(x,y,u) \right\}.
\end{eqnarray*}


Our aim is to establish that, as $\epsilon\to0^+$, the function~$V^\epsilon$ converges locally uniformly to a function~$V=V(t,x)$ (which will be independent of~$y$) which can be characterized as the unique (viscosity) solution to the {\it effective} Cauchy problem
\begin{equation}\label{EFFSP}
\left\{\begin{array}{ll}
-\partial_tV+\overline H\left(x,D_xV,D_{xx}^2V\right)+a V=0 &\quad \textrm{in }(0,T)\times \re^n\\
V(T,x)= \overline g(x)&\quad \textrm{on }\re^n.
\end{array}\right.
\end{equation}
The effective Hamiltonian and the effective terminal datum are given by
\begin{eqnarray}\label{opeff}
\overline H(x,p,X)&:=& \int_{\re^2}H(x,y,p,X,0) dm(y)\\ \label{datoeff}
\overline g(x)&:=& \int_{\re^2}g(x,y) dm(y)
\end{eqnarray}
and $m$ is the invariant measure established in Theorem~\ref{mgen}. As a matter of facts, $\overline H(x,p,X)$ is the ergodic constant~$\lambda$ of the cell problem
\begin{equation}\label{cell}
-tr(\sigma(y)\sigma^T(y)D^2 w(y))-b(y)Dw(y)+H(x,y,p,X,0)=\lambda\quad y\in\re^2,
\end{equation}
(the solution $w$ to this equation is called {\it corrector})
while $\overline g (x)$ is the constant obtained in the long time behaviour of the parabolic Cauchy problem
\begin{equation*}
\partial_t w^*- {\cal L} w^*=0\quad\textrm{in }(0,\infty)\times \re^2,\qquad w^*(0,y)=g(x,y)\quad\textrm{on } \re^2,
\end{equation*}
(namely $\overline g=\lim\limits_{t\to+\infty}w^*(t,y)$).

The main issues of this setting are: 1) the fast variables evolve in the whole space, 2) the infinitesimal generator of their operator is degenerate  with unbounded coefficients, 3) the variables $y$ lacks a group structure.
In order to overcome these issues, we shall use the following tools: 1) there exists a superlinear Lyapunov function,  2) a Liouville type result applies to operator ${\cal L}$, 3) there exists an invariant measure, 4) the cell problem admits a regular solution (we shall first prove that it is globally Lipschitz continuous and then we make a bootstrap argument) with an at most logarithmic growth. 

In order to prove the existence and the properties of $(\lambda, w)$ satisfying~\eqref{cell}, we introduce the approximated problems
\begin{equation}\label{probcellappross0}
\delta u_{\delta}-tr(\sigma(y)\sigma^T(y)D^2 u_{\delta})-b(y)Du_{\delta}=F(y)\qquad \textrm{in }\re^2,
\end{equation}
where  $\delta>0$ and $F(y):=-H(x,y,p,X,0)$ with $(x,p,X)$ fixed.
In the next subsection we investigate the properties of the approximated correctors~$u_{\delta}$; in the last subsection these properties will be inherited by the corrector~$w$.

\subsection{Regularity of the approximated correctors}
In this section we shall establish two results on the regularity of $u_\delta$ in two different setting for $F$: a global Lipschitz continuity  and a local H\"older continuity. In our opinion, both these results have their own interest because we apply two different techniques:  the former follows the ones of~\cite{CIL,IL} while the latter one follows the ones of~\cite{FIL}. However, in the rest of the paper we shall only need the former one.\\
Throughout this section we assume
\begin{equation}\label{L}
|F(y)|\leq C_F(|y|+1)\quad \forall y\in \re^2.
\end{equation}
Let us recall from \cite[Lemma 3.3]{MMT2} the following result on the growth of $u_\delta$; for the proof, we refer the reader to \cite{MMT2}.
\begin{lemma}\label{mmt2:l3.3}
Under assumptions~\eqref{L}, there exists a constant $C$ such that
\begin{equation}\label{sublindelta}
|u_\delta(y)|\leq C\bigg(|y|+ \frac{1}{\delta}\bigg), \quad y\in\re^2.
\end{equation}
\end{lemma}

\subsubsection{Global Lipschitz continuity of the approximated corrector}
\begin{proposition}\label{lipschitzNuovo}
Assume $b$ as in~\eqref{sigma} with $\alpha>1$ and that $F$ is Lipschitz continuous in $\re^2$ with Lipschitz constant~$L$. Let $u_\delta$ be the unique continuous solution of (\ref{probcellappross0}) which satisfies \eqref{sublindelta}. Then, for $\bar L>L/(\alpha-1)$, there holds
$$|u_\delta(x)-u_\delta(y)|\leq \bar L|x-y| \qquad \forall  x,y\in\re^2,\delta>0.$$
\end{proposition}
\begin{proof}
The proof follows the same arguments of the proof of \cite[Theorem 3.2]{MMT2}. For completeness, we briefly sketch the main steps.
For each $\eta>0$, we introduce the function
$$\Psi(x,y)=u_\delta(x)-u_\delta(y)-\bar L|x-y|-\eta|x|^2-\eta|y|^2.$$
Our statement is equivalent to the following inequality
\begin{equation}\label{claim}
\Psi(x,y)\leq \frac{4\eta}{\delta}\qquad \forall x,y \in\re^2,\, \eta\in(0,1).
\end{equation}
In order to prove \eqref{claim}, we argue by contradiction. Using the Lemma~\cite[Lemma 3.2]{CIL}, we follow the same calculation up to equation \cite[eq.(3.24)]{MMT2}. By our choice of the matrix~$\sigma$, we obtain the desired contradiction.
\end{proof}
\begin{remark}{\rm
As in \cite{MMT2}, for $b(x)=(-\alpha_1 x_1,-\alpha_2 x_2)$, we obtain the same result when $\alpha_1>1$, $\alpha_2>0$ and $\bar L>L/l$ where $l=\min\{\alpha_1-1,\alpha_2\}$.}
\end{remark}

\subsubsection{Local H\"older continuity of the approximated corrector}

\begin{proposition}
Assume $b$ as in~\eqref{sigma} with $\alpha>1$, \eqref{L} and
$$
|F(x)-F(y)|\leq C_F|x-y|^{\gamma}(\Phi(x)+\Phi(y)),\ x,y\in\re^2, \gamma\in(0,1], C_F>0$$
where $\Phi(x)= x_1^4+x_2^2+M$, $M\geq 1$.
Let $u_\delta$ be the unique continuous solution of (\ref{probcellappross0}) which satisfies \eqref{sublindelta}.
Then there is a constant $C>0$, independent on $\delta$ such that 
\begin{equation}\label{lipFIL}
|u_\delta(x)-u_\delta(y)|\leq C|x-y|^{\gamma}(\Phi(x)+\Phi(y)), \quad \forall\  x,y\in\re^2.
\end{equation}

\end{proposition}
\begin{proof}
We follow the procedure of \cite[Theorem 4.3]{FIL}.
We define the functions
$w_{\delta}(x,y)= u_\delta(x)-u_\delta(y)$  and 
 $\tilde g(x,y)= C_F|x-y|^{\gamma}(\Phi(x)+\Phi(y)+A)$ where $A$ will be chosen suitably large. If we  prove that $w_{\delta}\leq \tilde g$ in $\re^2$ then we obtain \eqref{lipFIL} 
 with a suitable $C$, since $\Phi > 1$.\\
 We argue by contradiction, we suppose that $\sup_{\re^2}(w_{\delta}-\tilde g)>0$. 
 From the linear growth of $u_\delta$ (see \eqref{sublindelta}) we know that
 $\lim_{|x|+|y|\to +\infty}(w_{\delta}(x,y)-\tilde g(x,y))=-\infty$ hence we have that $w_{\delta}\leq \tilde g $ in $(\re^2\times\re^2)\setminus B_R$ 
 for a suitable ball $B_R\subset\re^2\times\re^2$. 
 Let $(\hat x, \hat y)\in \overline{B_R}$ be maximum point  of $w_{\delta}(x,y)-\tilde g(x,y)$: $w_{\delta}(\hat x, \hat y)-\tilde g(\hat x, \hat y)>0$,
 $\hat x\neq \hat y$.
 At this point
we introduce the operator 
$\Xi$ defined as:
$$\Xi g(x,y)= tr(\Sigma(x,y)D^2g(x,y))+ 2\sum_{i,j}\frac{\partial^2}{\partial x_i\partial x_j}g(x,y)$$
where
$$
\Sigma(x,y)=\left(\begin{array}{cc}\sigma(x)\sigma^T(x) & \sigma(x)\sigma^T(y) \\\sigma(y)\sigma^T(x) & \sigma(y)\sigma^T(y)\end{array}
\right)= 
\left(\begin{array}{cccc}1 & 0 & 1 & 0 \\0 & x_1^2 & 0 & x_1y_1 \\1 & 0 & 1 & 0 \\0 & x_1y_1 & 0 & y_1^2\end{array}\right)
$$
(the matrix $\sigma$ is defined in \eqref{sigma}). This operator is elliptic.\\
We observe that $w_{\delta}(x,y)$ satisfies for any $x,y\in\re^2$
 \begin{eqnarray*}
  \delta w_{\delta} -\Xi w_{\delta}+\alpha x D_xw_{\delta}+\alpha yD_yw_{\delta} &=& F(x)-F(y)\\
  &\leq& C_F|x-y|^{\gamma}(\Phi(x)+\Phi(y)).
\end{eqnarray*}
Hence from the maximum principle
we have that 
\begin{equation}\label{abs}
 \delta \tilde g(\hat x, \hat y)-\Xi \tilde g(\hat x,\hat y)+\alpha \hat x D_x\tilde g(\hat x, \hat y)+\alpha yD_y\tilde g(\hat x, \hat y) \leq\\ 
C_F|\hat x-\hat y|^{\gamma}(\Phi(\hat x)+\Phi(\hat y)). 
\end{equation}
At this point to find a contradiction we compute 
$$\delta \tilde g(\hat x, \hat y)-\Xi \tilde g(\hat x,\hat y)+\alpha \hat x D_x\tilde g(\hat x, \hat y)+\alpha yD_y\tilde g(\hat x, \hat y)$$
directly by the definition of $\tilde g$.\\
 Denoting by $t=|x-y|^2$, let us introduce $g(x,y)$ as
 $$g(x,y)= t^{\gamma/2}(\Phi(x)+\Phi(y)+A).$$
 We compute now 
 $\Xi g(x,y)= tr(\Sigma(x,y)D^2g(x,y))$.
 We have
 $$D_xg= \gamma t^{\gamma/2-1}(x-y)(\Phi(x)+\Phi(y)+A)+t^{\gamma/2}D_x\Phi(x),$$
 $$D_yg= \gamma t^{\gamma/2-1}(y-x)(\Phi(x)+\Phi(y)+A)+t^{\gamma/2}D_y\Phi(y).$$
 \begin{eqnarray*}
 &&D^2_{xx}g= \gamma (\gamma-2)t^{\gamma/2-2}(x-y)\otimes (x-y)
 (\Phi(x)+\Phi(y)+A)+\\
 &&\gamma t^{\gamma/2-1}I(\Phi(x)+\Phi(y)+A)+2\gamma t^{\gamma/2-1}(x-y)\otimes D_x\Phi(x) +
 t^{\gamma/2}D^2_{xx}\Phi(x)
 \end{eqnarray*}
  \begin{eqnarray*}
 &&D^2_{yy}g= \gamma (\gamma-2)t^{\gamma/2-2}(y-x)\otimes (y-x)
 (\Phi(x)+\Phi(y)+A)+\\
 &&\gamma t^{\gamma/2-1}I(\Phi(x)+\Phi(y)+A)+2\gamma t^{\gamma/2-1}(y-x)\otimes D_y\Phi(y) +
 t^{\gamma/2}D^2_{yy}\Phi(y)
  \end{eqnarray*}
   \begin{eqnarray*}
&&D^2_{xy}g= -\gamma (\gamma-2)t^{\gamma/2-2}(x-y)\otimes (x-y)
 (\Phi(x)+\Phi(y)+A)-\\
&& \gamma t^{\gamma/2-1}I(\Phi(x)+\Phi(y)+A)+\gamma t^{\gamma/2-1}(x-y)\otimes(D_y\Phi(y)-D_x\Phi(x)).
  \end{eqnarray*}

Denoting by $A_{ij}$ the $2\times2$ minor of $\Sigma$ we have that
\begin{eqnarray*}
&&\Xi g(x,y)= tr(\Sigma(x,y)D^2g(x,y))=\\
&& tr(A_{11}D^2_{xx}g+A_{12}(D^2_{xy}g)^T+ A_{12}D^2_{xy}g+A_{22}D^2_{yy}g).
 \end{eqnarray*}
Using the explicit derivatives written here above and the definition of $\Phi$ we obtain
\begin{eqnarray*}
 &&\Xi g(x,y)= tr(\Sigma(x,y)D^2g(x,y))=\\ 
  &&\gamma (\gamma-2)t^{\gamma/2-2}(\Phi(x)+\Phi(y)+A)(x_2-y_2)^2(x_1-y_1)^2+\\
  &&t^{\gamma/2}(\Delta_{Gx}\Phi(x)+\Delta_{Gy}\Phi(y))+
 \gamma t^{\gamma/2-1}(\Phi(x)+\Phi(y)+A)(x_1-y_1)^2+\\
  &&+4\gamma t^{\gamma/2-1}(x_1-y_1)(x_2-y_2)(x_1x_2+y_1y_2)
  \end{eqnarray*}
  where we denoted by
  $\Delta_{G}u(z):=tr(\sigma(z)\sigma^T(z)u(z))$, i.e. the horizontal Grushin Laplacian operator.\\
 We note that, by elementary calculations,
it is possible to find a constant $L_{\alpha}$ such that 
\begin{equation}\label{theta}
-\Delta_{G}\Phi(z)+\alpha z D_z\Phi(z) \geq 2\alpha\Phi(z)-L_{\alpha}.
\end{equation}

  Now we write the equation
  \begin{eqnarray*}
  &&\delta g(x,y)-\Xi g(x,y)+\alpha x D_xg+\alpha yD_yg = \delta t^{\gamma/2}(\Phi(x)+\Phi(y)+A)+\\
  &&t^{\gamma/2}(-\Delta_{Gx}\Phi(x)+\alpha x D_x\Phi+ 
  (-\Delta_{Gy}\Phi(y)+\alpha y D_y\Phi))-\\
  && \gamma (\gamma-2)t^{\gamma/2-2}(\Phi(x)+\Phi(y)+A)(x_2-y_2)^2(x_1-y_1)^2\\
   && -\gamma t^{\gamma/2-1}(\Phi(x)+\Phi(y)+A)(x_1-y_1)^2-\\
    &&4\gamma t^{\gamma/2-1}(x_1-y_1)(x_2-y_2)(x_1x_2+y_1y_2)+
    \alpha \gamma t^{\gamma/2-1}(\Phi(x)+\Phi(y)+A)(x-y)^2\geq \\
   && \delta t^{\gamma/2}(\Phi(x)+\Phi(y)+A)+\\
    &&
   t^{\gamma/2}(2\alpha(\Phi(x)+\Phi(y))-2L_{\alpha})-\gamma t^{\gamma/2-1}(\Phi(x)+\Phi(y)+A)(x_1-y_1)^2-\\
   &&4\gamma t^{\gamma/2-1}(x_1-y_1)(x_2-y_2)(x_1x_2+y_1y_2)
   +\alpha \gamma t^{\gamma/2-1}(\Phi(x)+\Phi(y)+A)(x-y)^2,
  \end{eqnarray*}
  where in the last inequality we used \eqref{theta}.
 Hence by the definition of $t=|x-y|^2$ we have
 \begin{eqnarray*}
  &&\delta g(x,y)-\Xi g(x,y)+\alpha x D_xg+\alpha yD_yg \geq \\
  &&\delta |x-y|^{\gamma}(\Phi(x)+\Phi(y)+A)+\\
    &&
   |x-y|^{\gamma}(2\alpha(\Phi(x)+\Phi(y))-2L_{\alpha})-\gamma |x-y|^{\gamma-2}(\Phi(x)+\Phi(y)+A)(x_1-y_1)^2-\\
   &&4\gamma |x-y|^{\gamma-2}(x_1-y_1)(x_2-y_2)(x_1x_2+y_1y_2)+\\
   &&+\alpha \gamma |x-y|^{\gamma-2}I(\Phi(x)+\Phi(y)+A)(x-y)^2.
     \end{eqnarray*}
    Recall that $\tilde g(x,y)=C_Fg(x,y)$, hence $\tilde g$ satisfies:
  \begin{eqnarray*}
  &&\delta \tilde g(x,y)-\Xi \tilde g(x,y)+\alpha x D_x\tilde g+\alpha yD_y\tilde g \geq\\ 
&&C_F|x-y|^{\gamma}\bigg(2\alpha(\Phi(x)+\Phi(y))-2L_{\alpha}+(\delta-\gamma) (\Phi(x)+\Phi(y)+A)+\\
&&
+\alpha \gamma(\Phi(x)+\Phi(y)+A)-
   4\gamma(x_1x_2+y_1y_2)\bigg).
   \end{eqnarray*}
  Hence 
\begin{eqnarray*}
  &&\delta \tilde g(x,y)-\Xi \tilde g(x,y)+\alpha x D_xg+\alpha yD_y\tilde g \geq\\ 
&&C_F|x-y|^{\gamma}\bigg((\Phi(x)+\Phi(y))(\delta+(2\alpha -\gamma) +\gamma\alpha))+\\
&&A(\delta+(\alpha-1)\gamma)-
   2\gamma|x_1x_2+y_1y_2|-2L_{\alpha}\bigg).
   \end{eqnarray*}
   (We used that $|(x_1-y_1)(x_2-y_2)|\leq \frac{1}{2}(x-y)^2$.)
   Since $\alpha>1$ and $\gamma\in(0,1]$,
   \begin{eqnarray*}
  &&\delta \tilde g(x,y)-\Xi \tilde g(x,y)+\alpha x D_xg+\alpha yD_y\tilde g \geq\\ 
&&C_F|x-y|^{\gamma}\bigg((\Phi(x)+\Phi(y))+\\
&&\gamma\alpha(\Phi(x)+\Phi(y))+
A(\delta+(\alpha-1)\gamma)-
   4\gamma(x_1x_2+y_1y_2)-2L_{\alpha}\bigg)>\\ 
 &&C_F|x-y|^{\gamma}(\Phi(x)+\Phi(y)).
   \end{eqnarray*}
 The last inequality is obtained noting that:\\
 1) Since $\alpha>1$ we can find $K_{\alpha}>0$ such that 
 $\alpha\Phi(x)-4\gamma x_1x_2>-K_{\alpha}$.\\
 2) Since $\alpha>1$ we can choose $A$ sufficiently large such that $A(\delta+(\alpha-1)\gamma)-2L_{\alpha}-K_{\alpha}>0$.\\ 
Hence we obtain a contradiction of \eqref{abs}.
 \end{proof}
 \begin{remark}{\rm
 Note that if we consider a drift term of the type:
 $b(y)=(-\alpha_1y_1, -\alpha_2y_2)$ we can obtain the same result as before 
 taking $\gamma=1$ with $\alpha_1>1$ and $\alpha_2>0$. The calculations are tedious and we omit them.}
 \end{remark}

\subsection{The convergence result}\label{pbsp}

\begin{theorem}\label{maintheor}
Assume $\alpha>1$ and that, for $F(\cdot)=-H(x, \cdot, p, X,0)$
\begin{equation}\label{H}
\textrm{$F$, $\frac{\partial F}{\partial y_2}$ and $\frac{\partial^2 F}{\partial y_2^2}$ are bounded Lipschitz continuous functions.}
\end{equation}
Then, the solution $V^\epsilon$ of \eqref{HJBSP} converges locally uniformly in $(0,T)\times \re^n\times\re^2$ to the unique viscosity solution $V$ of \eqref{EFFSP} where $\overline H$ and $\overline g$ are defined in \eqref{opeff}-\eqref{datoeff}.
\end{theorem}

\begin{proof}
The arguments of the proof are analogous to those of \cite[Theorem 2.1]{MMT2}; we only sketch them.
\begin{enumerate}
\item Well posedness of problem \eqref{HJBSP} and growth properties of $V^\epsilon$.
\begin{proposition}\label{Existence}
For any $\epsilon>0$ there exists a unique continuous viscosity solution~$V^\epsilon$ to problem \eqref{HJBSP} such that
\begin{equation*}
|V^{\epsilon}(t,x,y)|\leq C_0(1+|x|), \quad \forall (t, x, y)\in (0,T)\times\re^n\times \re^2
\end{equation*}
for some positive constant~$C_0$ independent on $\epsilon$.
In particular $\{V^{\epsilon}\}_\epsilon$ is a family of locally equibounded functions.
\end{proposition}
\begin{proof}
The proof is the same as \cite[Proposition 2.1]{MMT2}.
\end{proof}
\item The cell problem.\\
Let us consider the sequence of solutions of the approximated cell problem \eqref{probcellappross0} $\{u_{\delta}\}_{\delta}$. Using the Proposition \ref{lipschitzNuovo} we can define (at least for a subsequence) the 
$ \lim\limits_{\delta\rightarrow 0}(u_{\delta}(y)-u_{\delta}(0)):=w(y)$ and using the Lemma \ref{mmt2:l3.3} $ \lim\limits_{\delta\rightarrow 0}-\delta u_{\delta}(0)=:\lambda$. 
\\
Thanks to  Proposition \ref{lipschitzNuovo} $w$ is a global Lipschitz function and using the stability properties of viscosity solutions $(w,\lambda)$ is a solution of the ergodic problem \eqref{cell}.\\ Moreover:
\begin{proposition}\label{TH33}
The constant $\lambda=-\int_{\re^2} H(\overline x,y,\overline p, \overline X,0)dm(y)$ ($m$ is the invariant measure founded in Theorem \ref{mgen}) is the unique constant  such that the cell problem \ref{cell} admits a solution~$w$ with an at most linear growth at infinity. Moreover $w$ is globally Lipschitz continuous, satisfies
\begin{equation}\label{wlog}
|w(y)-w(0)|\leq C\left[1+\log(y_1^4+y_2^2+1)\right] \qquad \forall y\in \re^2
\end{equation}
and it is unique up to an additive constant within the function with an at most linear growth at infinity.
\end{proposition}
\begin{proof}
We refer the reader to Remark \ref{ergodics} to characterize $\lambda$.\\
Estimate \eqref{wlog} follows from an analogous estimate for $u_{\delta}$ that can be proved as in \cite[Lemma 3.4]{MMT2} taking as supersolution of \eqref{probcellappross0} the function
$g(y_1,y_2)=C_1\log(y_1^4+y_2^2)$ which satisfies
$$
\delta g-tr(\sigma\sigma^tD^2g)+\alpha yDg\geq C_1\bigg(\frac{2y_1^6-10y_1^2y_2^2}{(y_1^4+y_2^2)^2}+\alpha\frac{4y_1^4+y_2^2}{y_1^4+y_2^2}\bigg)\geq F(y),$$
for $y\in\re^2\setminus B_R$,
with suitable $C_1$ and $R$ sufficiently large.
Hence repeating the same argument as in \cite[Lemma 3.4]{MMT2} we get the result.
\end{proof}
\item $C^2$-regularity of the corrector.
\begin{proposition}\label{regw}
Let $w$ be the solution of the cell problem \eqref{cell} founded in Proposition~\ref{TH33}. Then $w\in C^{2,\beta}_{loc}(\re^2)$, for some $\beta\in(0,1)$.
\end{proposition}
\begin{proof}
In this proof, $\beta$ denotes a constant which may change from line to line.
The corrector $w$ solves 
\begin{equation}\label{eqconG}
-tr(\sigma(y)\sigma^T(y)D^2 w(y))+\alpha yDw(y)=G(y)
\end{equation}
with $G(y):=\lambda-H(\overline x,y,\bar p,\bar X,0)$.
First let us get the global Lipschitz continuity of $\frac{\partial w}{\partial y_2}$ and $\frac{\partial^2 w}{\partial y_2^2}$ .
Deriving equation \eqref{eqconG} with respect to $y_2$ (remark that this is possible because $G$ is regular enough thanks to \eqref{H}) we obtain that the function $u:=\frac{\partial w}{\partial y_2}$ is bounded by Proposition~\ref{lipschitzNuovo} and it solves in the sense of distributions
 \begin{equation}\label{equ}
-tr(\sigma\sigma^TD^2 u)+\alpha y Du+\alpha u=\frac{\partial G}{\partial y_2}.
\end{equation}
From Proposition~\ref{lipschitzNuovo} and \cite[Lemma 3.5]{MMT2}, we get that $\frac{\partial w}{\partial y_2}$ is globally Lipschitz continuous in $\re^2$.
Deriving again equation \eqref{equ} with respect to $y_2$ we obtain that the function $\frac{\partial^2 w}{\partial y_2^2}$  is globally Lipschitz continuous in $\re^2$.

Using the global Lipschitz continuity of $\frac{\partial w}{\partial y_2}$ and  Proposition~\ref{lipschitzNuovo} in \eqref{eqconG}, we infer: $\frac{\partial^2 w}{\partial y_1^2}\in L^\infty_{loc}$. Again, by the Lipschitz continuity of $\frac{\partial w}{\partial y_2}$, we obtain $\Delta w \in L^\infty_{loc}$; by standard elliptic theory, $Dw\in C^{0,\beta}_{loc}$. Using the global Lipschitz continuity of $\frac{\partial^2 w}{\partial y_2^2}$ in \eqref{eqconG}, we get $\frac{\partial^2 w}{\partial y_1^2}\in C^{0,\beta}_{loc}$. Again, by the Lipschitz continuity of $\frac{\partial^2 w}{\partial y_2^2}$, we have $\Delta w \in C^{0,\beta}_{loc}$. Applying standard theory, we accomplish the proof.
\end{proof}
\item Conclusion. \\
We adapt the classical perturbed test function method (see \cite{AB1,Ev,G2}) to prove the convergence. To this end, we argue as in \cite[Theorem 2.1]{MMT2} using the Liouville property for ${\cal L}$, the regularity of the corrector and the existence of a Lyapunov function ($W(y)=y_1^2+y_2^2$).
\end{enumerate}
\end{proof}

\noindent{\sc Acknowledgments. } \\
The first and the second authors  are members of the INDAM-Gnampa and are partially supported by the research project of the University of Padova "Mean-Field Games and Nonlinear PDEs" and by the Fondazione CaRiPaRo Project "Nonlinear Partial Differential Equations: Asymptotic Problems and Mean-Field Games". 
The third author has been partially funded by the ANR project ANR-16-CE40-0015-01.
The  third author wishes to thank for the hospitality the Department of Mathematics, University of Padova, where part of this work was done.

\vskip 0.3truecm

\noindent {\bf Address of the authors}\\
Paola Mannucci,\\
Dipartimento di Matematica "Tullio Levi-Civita",
Universit\`a degli Studi di Padova,
Via Trieste 63, 35131, Padova, Italy,\\
Claudio Marchi,\\
Dipartimento di Ingegneria dell'Informazione,
Universit\`a degli Studi di Padova,
Via Gradenigo 6/b, 35131, Padova, Italy,\\
Nicoletta Tchou, IRMAR, \\
Universit\`a de Rennes 1,
Campus de Beaulieu, 35042 Rennes Cedex, France. \\ \\
mannucci@math.unipd.it,\\ 
claudio.marchi@unipd.it,\\
nicoletta.tchou@univ-rennes1.fr

\end{document}